\documentclass[10pt,twoside,reqno]{amsart}

\newcommand{\vs}{\vspace}
\newcommand{\bc}{\begin{center}}
\newcommand{\ec}{\end{center}}
\usepackage{amsmath}
\usepackage{amssymb}
\usepackage{amsfonts}
\usepackage{amsthm}

\theoremstyle{plain}
\newtheorem{theorem}{Theorem}[section]
\newtheorem{lemma}[theorem]{Lemma}
\newtheorem{proposition}[theorem]{Proposition}
\newtheorem{corollary}[theorem]{Corollary}
\theoremstyle{definition}
\newtheorem{definition}[theorem]{Definition}
\newtheorem{remark}[theorem]{Remark}
\newtheorem{example}[theorem]{Example}
 \numberwithin{equation}{section}

\begin{document}
\setcounter{page}{1}

\title[Perimetric Contraction on Polygons and Related Fixed Point Theorems]{\large Perimetric Contraction on Polygons and Related Fixed Point Theorems}
\author[]{Mi Zhou$^{*}$, Evgeniy Petrov$^{**,\dag}$}
\date{}
\maketitle

\vs*{-0.5cm}

\bc
{\footnotesize
$^{*}$Center for Mathematical Reaserch, University of Sanya, Sanya, Hainan 572022, China.\\
E-mail: mizhou@sanyau.edu.cn\\
\medskip
$^{**}$Function Theory Department, Institute of Applied Mathematics and Mechanics of the NAS of Ukraine, Batiuka Str. 19, Slovyansk 84116, Ukraine.\\
E-mail: eugeniy.petrov@gmail.com\\
$^\dag$ Corresponding Author: Evgeniy Petrov}
\ec

\bigskip

{\footnotesize
\noindent
{\bf Abstract.}
In the present paper, a new type of mappings called perimetric contractions on $k$-polygons is introduced. These contractions can be viewed as a generalization of mappings that contracts perimeters of triangles.
A fixed point theorem for this type of mappings in a complete metric space is established. Achieving a fixed point necessitates the avoidance of periodic points of prime period $2,3,\cdots, k-1$. The class of contraction mappings is encompassed by perimeter-based mappings, leading to the recovery of Banach's fixed point theorem as a direct outcome from our main result. A sufficient condition to guarantee the uniqueness of the fixed point is also provided. Moreover, we introduce the Kannan type perimetric contractions on $k$-polygons, establishing a fixed point theorem and a sufficient uniqueness condition. The relationship between these contractions, generalized Kannan type mappings, and mappings contracting the perimeters on $k$-polygons is investigated. Several examples are illustrated to support the validity of our main results.



}

\bigskip

\section{Introduction }
\quad Fixed point theory is essential in mathematics since many problems can be viewed as fixed point problems, which focus on determining the existence and uniqueness of solutions. Its applications are broad, covering matrix equations, differential equations, integral equations, optimization, image process and machine learning. The field's seminal work traces back to Stefan Banach's (1922) introduction of the Banach contraction principle \cite{Banach}, ensuring unique fixed points for contraction mappings in complete metric spaces. Subsequently, other notable scholars made substantial contributions to the advancement of fixed point theory. The enduring interest of mathematicians in fixed-point theorems remains strong after a century, as evidenced by the proliferation of articles and monographs in recent decades focusing on the theory and its applications, as discussed in various works, see e.g.~\cite{Kirk}-\cite{Subrahmanyam}. The Banach contraction principle has been generalized in many ways over the years. In fact, these generalizations of Banach contraction can mainly be categorized into four main groups, each distinguished by its unique type of extension. The first category involves relaxing the mapping's contractive condition by adding compatible distances between points or control functions. The second category broadens by changing the topology or expanding metric space definitions. The third category extends the theorem's applicability to multi-valued mappings, exploring fixed points in a generalized setting. Lastly, the fourth category studies common fixed points and best proximity results for mappings in generalized metric or topological spaces. Please refer to \cite{Reich}-\cite{Debnath} for more information on extensions of contractive mappings in various settings.

Let $X$ be a metric space and $T$ be a self mapping defined on $X$. Usually, Banach contraction or other conditions typically involve distances between two points, focusing on images produced by the operator $T$ and its original preimages, like $d(x, y), d(Tx, Ty), d(x, Tx), d(y, Ty), d(x, Ty), d(y, Tx)$, and more. Pittnauer \cite{Pittnauer} and Achari \cite{Achari} explored fixed point theorems for contractive type mappings involving three points in the space. While some authors also examined three or four points, their focus was on up to six distances in a standard contractive definition using different combinations of four points taken two at a time.

In 2023, Petrov \cite{Petrov1} introduced a new class of mappings that contract perimeters of triangles and extended Banach contraction. The main theorem in this work is proven using concepts from Banach's classical theorem, with a key distinction being that the provided mappings are defined based on three points in space instead of two. In addition, he imposed a condition to avoid the occurrence of periodic points of prime period 2 in the mapping $T$. The ordinary contraction mappings form an important subclass of such mappings.
\begin{definition}\label{def:1.1}\cite{Petrov1} Let $(X,d)$ be a metric space with at least three points. Then
the mapping $T: X\rightarrow X$ is defined as contracting perimeters of triangles if there is an $\alpha\in[0,1)$ such that
\begin{eqnarray}\label{eq:1.1}
d(Tx, Ty)+d(Tx,Ty)+d(Tz,Tx)\leq\alpha (d(x,y)+d(y,z)+d(z,x))
\end{eqnarray}
for three pairwise distinct points $x,y,z\in X$.
\end{definition}

\begin{remark}\label{rm:1.1}The prerequisite for $x,y,z\in X$ to be pairwise distinct is crucial in this definition. Without this condition, the definition coincides with that of a contraction mapping. Research has demonstrated the continuity of mappings contracting perimeters of triangles. When every point in the metric space serves as an accumulation point, the distinctions between these mappings and contraction mappings vanish.
\end{remark}

In 2024, Petrov and Bisht \cite{Petrov2} introduced a three-point analogue of Kannan type mappings \cite{Kannan} by utilizing the concept of mapping contracting perimeters of triangles, leading to the development of fixed point results.
\begin{definition}\label{def:1.2}\cite{Petrov2} Let $(X,d)$ be a metric space with at least three points. Then $T: X\rightarrow X$ is a generalized Kannan type mapping on $Y$ if there is a $\lambda\in[0,\frac{2}{3})$ such that
\begin{eqnarray}\label{eq:1.2}
d(Tx, Ty)+d(Ty,Tz)+d(Tz, Tx)\leq \lambda(d(x,Tx)+d(y,Ty)+d(z,Tz))
\end{eqnarray}
for any three pairwise distinct points $x,y,z\in X$.
\end{definition}
In a complete metric space, every generalized Kannan type mapping attains fixed points if it does not achieve periodic points of prime period 2. There are at most two fixed points. Moreover, in \cite{Petrov2}, the authors examined the relationships between generalized Kannan type mappings, Kannan type mappings, and mappings contracting perimeters of triangles. They concluded that the classes of Kannan type mappings and generalized Kannan type mappings are distinct. Additionally, they found that generalized Kannan type mappings are discontinuous in general case but continuous at fixed points.

Very recently, based on \cite{Petrov1}, Anis Banerjee \textit{et al.} \cite{Anis} introduced a new four-point version of Banach-type, Kannan-type, and Chatterjea-type contractions, called a perimetric contraction on quadrilaterals. They also analysed their characteristics and proved conditions for fixed point existence in a complete metric space.

We are intrigued by the mentioned findings and aim to investigate more than four-point analogue of prior findings to establish conditions for the existence and uniqueness of fixed points. Our objective also includes comparing various classes of mappings to reveal potential relationships between them.

In the second section, we introduce a novel type of mappings contracting the perimeters of $k$-polygons $(k\geq3,k\in\mathbb{N})$. We also discuss some key properties of these mappings. Furthermore, we establish a fixed point theorem for this type of mapping in a complete metric space. It is essential to avoid periodic points of prime periods $2,3,\ldots,k-1$ in order to obtain a fixed point. As a result, Banach's fixed point theorem can be directly applied. Additionally, we derive a sufficient condition for the fixed point to be unique. To support our findings, we provide practical examples.

In the third section, we introduce Kannan type perimetric contractions on $k$-polygons and establish a fixed point result. We also derive a sufficient condition for the fixed point to be unique. Furthermore, we investigate the relationship between Kannan type perimetric contractions on $k$-polygons, generalized Kannan type mappings, and mappings that contract the perimeters of $k$-polygons. Our findings show that these classes are distinct, and we provide non-trivial examples to illustrate this.

Throughout the next whole discussions, we denote $(X,d)$ as metric space, $|X|$ as the cardinality of the set $X$, $\mathbb{N}$ as the set of natural numbers. The concept of a periodic point is defined as follows. Let $T$ be a mapping on the metric space $X$. A point $x\in X$ is said to be a periodic point of period $p$ if $T^px=x$. The prime period of $x$ is the least positive integer $p$ for which $T^px=x$.

\section{Perimetric Contraction on $k$-Polygons and Related Fixed Point Theorem}
The results proved in~\cite{Petrov1} were generalized for mappings contracting total pairwise distances in~\cite{Petrov3}. Let $(X,d)$ be a metric space, $|X|\geqslant 2$, and let $x_1$, $x_2$, \ldots, $x_k \in X$, $k\geqslant 2$. Denote by
\begin{equation}\label{e0}
S(x_1,x_2,\ldots,x_k)=\sum\limits_{1\leqslant i<j\leqslant k}d(x_i,x_j)
\end{equation}
the sum of all pairwise distances between the points from the set $\{x_1, x_2, \ldots, x_k\}$, which we call \emph{the total pairwise distance}. For $k\geqslant 3$ denote also by
\begin{equation}\label{e00}
P(x_1,x_2,\ldots,x_k)=d(x_1,x_2)+d(x_2, x_3)+\cdots+d(x_{k-1},x_k)+d(x_k,x_1)
\end{equation}
the perimeter of a polygon on the \textbf{consecutive points} $x_1, x_2, \ldots, x_k$.

\begin{definition}[\cite{Petrov3}]\label{d1}
Let $k\geqslant 3$, $k\in \mathbb N$, and let $(X,d)$ be a metric space with $|X|\geqslant k$. We shall say that $T\colon X\to X$ is a \emph{mapping contracting total pairwise distance on $k$ points} if there exists $\lambda\in [0,1)$ such that the inequality
  \begin{equation}\label{e1}
S(Tx_1,Tx_2,\ldots,Tx_k) \leqslant \lambda S(x_1,x_2,\ldots,x_k)
  \end{equation}
holds for all $k$ pairwise distinct points $x_1, x_2, \ldots, x_k \in X$.
\end{definition}

The mappings called the perimetric contractions on $k$-polygons, introduced in the following definition,  is the main object of investigation of this section.

\begin{definition}\label{def:2.1}
Let $k\geqslant3$, $k\in \mathbb N$, and let $(X,d)$ be a metric space with $|X|\geqslant k$. We shall say that $T\colon X\to X$ is a \emph{perimetric contraction on $k$-polygons} in $X$ if there exists $\lambda \in [0,1)$ such that the inequality
  \begin{equation}\label{eq:2.1}
P(Tx_1,Tx_2,\ldots,Tx_k) \leqslant \lambda P(x_1,x_2,\ldots,x_k)
  \end{equation}
  holds for all $k$ pairwise distinct points $x_1, x_2, \ldots, x_k \in X$.
\end{definition}

\begin{remark}\label{rm:2.1} If we choose $k=3$ in Definition \ref{def:2.1}, then the mapping $T$ reduces to be the perimetric contraction on triangles considered in~\cite{Petrov1}, which in turn coincides with the mappings contracting total pairwise distance, see Definition~\ref{d1}.
\end{remark}

\begin{theorem}\label{th:2.1}
Let $(X,d)$ be a metric space with $|X|\geqslant3$ and let $3\leqslant k \leqslant |X|$, $k\in \mathbb N$. Then any perimetric contraction on $k$-polygons is a mapping contracting total pairwise distances on $k$ points with the same coefficient of contraction.
\end{theorem}
\begin{proof}
Let $x_1, x_2, \ldots, x_k$ be pairwise distinct points in $X$. The number of different polygons that can be formed on these  $k$ points is given by the formula for Hamiltonian cycles in a complete graph. Specifically, the number of distinct Hamiltonian cycles in an undirected complete graph with $k$ vertices is calculated as $H(k)=(k-1)!/2$. (There are $k!$ ways to arrange $k$ vertices in a cycle, but since the cycle can be traversed in two directions and can start at any of the $k$ vertices, we divide $k!$ by $2k$ to take into account these repetitions.)

For every of $H(k)$ $k$-polygons on the points $x_1, x_2, \ldots, x_k$ consider $H(k)$ inequalities \eqref{eq:2.1}. Summarizing the left and the right parts of these inequalities we get
\begin{equation}\label{e111}
\sum\limits_{i=1}^{H(k)} P(T\pi_i(x_1),T\pi_i(x_2),\ldots,T\pi_i(x_k)) \leqslant \alpha \sum\limits_{i=1}^{H(k)} P(\pi_i(x_1),\pi_i(x_2),\ldots,\pi_i(x_k)),
  \end{equation}
where by $\pi_i$ we denote the admissible permutations of the set $\{x_1, x_2, \ldots, x_k\}$.

Clearly, every edge $\{x_i,x_j\}$, $i\neq j$ in all these $H(k)$ $k$-polygons appears a fixed number of times depending only on $k$. To determine this number consider that each Hamiltonian cycle can be formed by fixing one edge and permuting the remaining vertices. For any given edge, there are $(k-2)!$ ways to arrange the remaining $k-2$ vertices in the cycle. Thus, the total number of Hamiltonian cycles that include a specific edge is $E(k)=(k-2)!$.

Dividing both parts of inequality \eqref{e111} by $E(k)$ we obtain exactly inequality \eqref{e1}, which completes the proof.
\end{proof}

The following example shows that perimetric contractions on $k$-polygons is a proper subclass of mappings contracting total pairwise distances on $k$ points.

\begin{example}\label{em:2.1}
Let $(X,d)$ be a metric space such that $X=\{x_1,x_2,x_3,x_4\}$, $d(x_1,x_2)=d(x_2,x_3)=d(x_3,x_4)=d(x_4,x_1)=d(x_1,x_3)=2$ and $d(x_2,x_4)=1$. Let also $T\colon X\to X$ be such that $T(x_1)=x_1$, $Tx_2=x_3$, $Tx_3=x_4$ and $Tx_4=x_1$.  One can see that $S(x_1,x_2,x_3,x_4)=11$, $S(Tx_1,Tx_2,Tx_3,Tx_4)=10$, $P(x_1,x_2,x_4,x_3)=7$ and $P(Tx_1,Tx_2,Tx_4,Tx_3)=8$. Thus, $T$ is a mapping contracting the total pairwise distance on $4$ points but not a perimetric contraction on $4$-polygons.
\end{example}

The following proposition was proved in~\cite{Petrov3}.
\begin{proposition}\label{p2.1}
Mapping contracting total pairwise distance on $m$ points, $m\geqslant 2$, is a mapping contracting total pairwise distance on $n$ points for all $n>m$.
\end{proposition}
Using Theorem~\ref{th:2.1} and Proposition~\ref{p2.1}, we get the following.

\begin{corollary}
Let $(X,d)$ be a metric space with $|X|\geqslant3$ and let $3\leqslant k \leqslant |X|$, $k\in \mathbb N$. Then any perimetric contraction on $k$-polygons is a mapping contracting total pairwise distances on $n$ points with the same coefficient of contraction for all $n>k$.
\end{corollary}

Theorem 2.4 from \cite{Petrov1} verify the assertion ``$T$ has a fixed point if and only if $T$ does not posses periodic points of prime period $2$''. Now, we will provide an example to show that the converse analogous implication is not applicable for perimetric contraction on $k$-polygons, in the case $k=7$.

\begin{example}\label{ex:2.1} Let $(X,d)$ be a metric space such that $X=\{x_1,x_2,\ldots,x_7\}$ and the metric $d$ defined as
\begin{eqnarray*}
d(x_i,x_j)=1, \, 1\leq i,j\leqslant 6, \, i\neq j \quad \text{and}\quad d(x_i,x_7)=2, \, 1\leq i\leqslant 6.
\end{eqnarray*}
Define $T$ as follows: $Tx_1=x_1$, $Tx_2=x_3$, $Tx_3=x_2$, $Tx_4=x_5$, $Tx_5=x_6$, $Tx_6=x_4$, and $Tx_7=x_1$. Then $P(x_1,x_2,\ldots,x_7)=1+1+1+1+1+2+2=9$, and $P(Tx_1,Tx_2,\ldots,Tx_7)=1+1+1+1+1+1+0=6$.
Note also that always $P(\pi(x_1),\pi(x_2),\ldots,\pi(x_7))=1+1+1+1+1+2+2=9$, where $\pi$ is any permutation on the set $X$ and $P(T\pi(x_1),T\pi(x_2),\ldots,T\pi(x_7))<9$ since this sum does not possess distances which are equal to $2$. Thus, $T$ is a perimetric contraction on $7$-polygons, also $T$ is not a perimetric contraction on 3-polygons as $P(Tx_4,Tx_5,Tx_6)=P(x_4,x_5,x_6)=3$. Moreover, $T$ has one fixed point $x_1$, two points $x_2$ and $x_3$ of prime period 2, three points $x_5, x_6$ and $x_7$ of prime period 3.
\end{example}

By virtue of the Theorem~\ref{th:2.1} Proposition~\ref{pro:2.1}, Theorem~\ref{th:2.2}, Proposition~\ref{pro:2.2}, Lemma~\ref{lem:2.3}, Proposition~\ref{cor:2.2}, are direct corollaries of the corresponding results proved in~\cite{Petrov3}. Nevertheless, below we give independent proofs of these results.

Recall that an accumulation point $x$ in a metric space $X$ is a point such that every open ball centered at $x$ contains infinitely many points of $X$.
\begin{proposition}\label{pro:2.1} Every perimetric contraction on $k$-polygons is continuous.
\end{proposition}
\begin{proof}
Let $(X,d)$ be a metric space with $|X|\geqslant3$, and a mapping $T:X\rightarrow X$ be a mapping contracting perimeters of $k$-polygons $(3\leqslant k\leq|X|, k\in\mathbb{N})$ on $X$. Chosen any $x^\ast\in X$, we consider the following two potential cases. If $x^\ast$ is an isolated point in $X$, then $T$ is continuous at $x^\ast$. If $x^\ast$ is not an isolated point but an accumulation point. Hence, the rest proof is to prove that for any $\epsilon>0$, there exits $\delta>0$ such that $d(Tx,Tx^\ast)<\epsilon$ whenever $d(x,x^\ast)<\delta$.
For any $\epsilon>0$, choose $\delta>0$ being such that $0<\delta<\frac{\epsilon}{2(k-1)}$.
Since $x^\ast$ is an accumulation point, there exist $p_i\in X, i=1,2,\ldots,k-2$ such that $d(p_i,x^\ast)<\delta, i=1,2,\ldots,k-2$. Now, for any $x\in X$ with $x\neq x^\ast$ satisfying $d(x,x^\ast)<\delta$, we have
\begin{align*}
d(Tx,Tx^\ast)&\leq P(Tx,Tx^\ast,Tp_1,Tp_2,\ldots,Tp_{k-3},Tp_{k-2})\\
&\leq\lambda P(x,x^\ast,p_1,p_2,\ldots,p_{k-3},p_{k-2})\\
&\leq2\lambda(d(x,x^\ast)+d(x^\ast,p_1)+d(p_2,x^\ast)+\cdots+d(p_{k-2},x^\ast))\\
&<2(k-1)\delta\\
&<\epsilon.
\end{align*}
Therefore, the conclusion follows.
\end{proof}

We shall now establish the requisite condition for the existence of fixed points in perimetric contractions on $k$-polygons.

\begin{theorem}\label{th:2.2} Suppose $(X, d)$ is a complete metric space with $|X|\geqslant 3$. Let $T: X\rightarrow X$ be a perimetric contraction on $k$-polygons $(3\leqslant k\leq|X|, k\in\mathbb{N})$ in $X$. $T$ has a fixed point in $X$ if it does not have periodic points of prime periods $i, i=2,3,\ldots,k-1$, and it can admit at most $k-1$ fixed points.
\end{theorem}
\begin{proof}
Let $T: X\rightarrow X$ be a perimetric contraction on $k$-polygons in $X$ that does not have periodic points of prime period $i, i=2,3,\cdots,k-1$.
For any chosen $x_0\in X$, define the sequence $\{x_n\}$ by $x_n=Tx_{n-1}, n\in\mathbb{N}$. If $x_n$ is a fixed point of $T$ for any $n\in\mathbb{N}\cup\{0\}$, then the proof is completed. Assume that $x_n\neq Tx_n$ for all $n\in\mathbb{N}\cup\{0\}$, we have $x_n\neq x_{n+1}, n=0,1,2,\cdots$. Since $T$ does not attain periodic points of prime periods $2,3,\cdots,k-1$, therefore, it follows from a simple computation that every $k$ consecutive elements of $\{x_n\}$ are pairwise distinct.

Let $r_n=P(x_n,x_{n+1},x_{n+2},\ldots,x_{n+k-2},x_{n+k-1})$ for all $n\in\mathbb{N}\cup\{0\}$, then $r_n>0$, for all $n\in\mathbb{N}\cup\{0\}$.\\
From the perimetric contraction assumption \eqref{eq:2.1}, for all $n\in\mathbb{N}$ we have $r_n\leq \lambda r_{n-1}$. Also,
\begin{align*}
&d(x_0,x_1)\leq r_0,\\
&d(x_1,x_2)\leq r_1\leq \lambda r_0,\\
&\quad\quad\vdots\\
&d(x_n,x_{n+1})\leq r_n\leq\lambda r_{n-1}\leq \lambda^n r_0.
\end{align*}
Now, for any $n\in\mathbb{N}\cup\{0\}$ and any $m\in\mathbb{N}$, we have
\begin{align*}
d(x_n,x_{n+m})&\leq d(x_{n},x_{n+1})+d(x_{n+1},x_{n+2})+\cdots+d(x_{n+m-1},x_{n+m})\\
&\leq \lambda^n r_0+\lambda^{n+1}r_0+\cdots+\lambda^{n+m-1}r_0\\
&=\lambda^n(1+\lambda+\lambda^2+\cdots+\lambda^{m-1})r_0\\
&\leq \lambda^n\frac{1}{1-\lambda}r_0,
\end{align*}
which shows that $\{x_n\}$ is a Cauchy sequence. Due to the completeness of $X$, it follows that $x_n$ converges to a point $w$ in $X$.

Let us prove that $Tw=w$. Since $x_n\to w$, and by Proposition~\ref{pro:2.1} the mapping $T$ is continuous, we have $x_{n+1}=T x_n\to Tw$.  By the triangle inequality we have
$$
d(w,Tw)\leqslant d(w,x_{n})+d(x_{n},Tw)\to 0\,\, \text{ as } n\to \infty,
$$
which means that $w$ is the fixed point and contradicts to our assumption.

Assume that $T$ have at least $k$ distinct fixed points, say $w_i, i=1,2,\ldots,k$, that is, $Tw_i=w_i, i=1,2,\ldots,k$. Then
\begin{eqnarray*}
P(Tw_1,Tw_2,\ldots,Tw_{k-1},Tw_k)\leq\lambda P(w_1,w_2,\ldots,w_{k-1},w_k),
\end{eqnarray*}
which implies that $\lambda\geq 1$, a contradiction to \eqref{eq:2.1}. Hence, the conclusion follows.
\end{proof}

We now offer an example supporting Theorem~\ref{th:2.2}. It  illustrates a mapping that contracts the perimeters of $k$-polygons, $k\in\mathbb{N}$, $k\geqslant 3$, while possessing $k-1$ fixed points.

\begin{example}\label{em:2.2} Consider the metric space $(X,d)$ where $X=\{p_1,p_2,\ldots,p_k\}, k\geqslant3$ and $d$ is such that
\begin{eqnarray*}
d(x_i,x_j)=1, \, 1\leq i,j\leqslant k-1, \, i\neq j, \quad \text{and}\quad d(x_i,x_k)=2, \, 1\leq i\leqslant k-1.
\end{eqnarray*}
The mapping $T:X\rightarrow X$ defined as $Tp_1=p_1, Tp_2=p_2, \cdots,Tp_k=p_1$, forms a perimetric contraction on the $k$-polygons in $X$, see Example~\ref{ex:2.1} for the corresponding analogous calculations. Additionally, $T$ does not have periodic points of prime periods $2,3,\ldots,k-1$ ensuring the existence of a fixed point for $T$, as indicated by Theorem \ref{th:2.1}. It is evident that the fixed points of $T$ are $\{p_1,p_3,\ldots, p_{k-1}\}$.
\end{example}

Observe from Example~\ref{em:2.2} that perimetric contractions on $k$-polygons may have multiple fixed points. To ensure a unique fixed point for this mapping, an infinite complete metric space is considered, leading to the subsequent result.

\begin{proposition}\label{pro:2.2} Suppose that under the assumption of Theorem \ref{th:2.2}, the mapping $T$ has a fixed point $x^\ast$ that acts as the limit for a specific iteration sequence $\{x_i\}_0^\infty$ defined by $x_i=Tx_{i-1}, i\in\mathbb{N}$ with $x^\ast\neq x_i$ for all $i\in\mathbb{N}\cup\{0\}$, then $x^\ast$ is the unique fixed point of $T$.
\end{proposition}
\begin{proof} Suppose that $w^\ast$ is another fixed point of $T$ with $x^\ast\neq w^\ast$. It is clear that $x_i\neq w^\ast$ for $i\in\mathbb{N}\cup\{0\}$, otherwise, we have $w^\ast=x^\ast$. Hence, $w^\ast, x^\ast$ and $x_i, i\in\mathbb{N}\cup\{0\}$ are pairwise distinct points.
Consider the ratio
\begin{align*}
K_i&=\frac
{P(Tx^\ast,Tw^\ast,Tx_i,Tx_{i+1},Tx_{i+2},\ldots,Tx_{i+k-3})}
{P(x^\ast,w^\ast,x_i,x_{i+1},x_{i+2},\ldots,x_{i+k-3})}\\
&=\frac{P(x^\ast,w^\ast,x_{i+1},x_{i+2},x_{i+3},\ldots,x_{i+k-2})}
{P(x^\ast,w^\ast,x_i,x_{i+1},x_{i+2},\ldots,x_{i+k-3})}.
\end{align*}
Then by~\eqref{eq:2.1} we have  $K_i\leq \lambda$ for all $i\in\mathbb{N}\cup\{0\}$. Taking the limit in the above inequality as $i\rightarrow\infty$, we obtain that $K\rightarrow 1$, which contradicts to \eqref{eq:2.1}. Therefore, $T$ has a unique fixed point.
\end{proof}

Recall that a self mapping defined on a metric space $(X,d)$ is a contraction, if there exits $\lambda\in[0,1)$ such that
\begin{eqnarray}\label{eq:2.2}
d(Tx,Ty)\leq \lambda d(x,y), \, \, \text{ for all } \, x,y\in X.
\end{eqnarray}

Next, we will present an alternative proof of Banach Contraction Principle using Theorem \ref{th:2.2} in the following corollary.

\begin{corollary}\label{cor:2.1}(Banach Contraction Principle) A self mapping $T$ defined on a completed metric space $(X,d)$ being a contraction has a unique fixed point in $X$.
\end{corollary}
\begin{proof} If $|X|=1,2$, the conclusion follows from the argument stated in the proof of Corollary 2.6 in \cite{Petrov1}. Let $|X|\geqslant 3$. xIt is obvious that $T$ has no periodic points of prime periods $2,3,\ldots,k-1$ for any $k\geqslant 3$, otherwise it contradicts to the Banach contraction condition. Indeed, for any periodic point $p(r)$ of prime period $r\in\{2,3,\ldots,k-1\}$, then
\begin{eqnarray*}
d(p(r),Tp(r))=d(T^rp(r),T^{r+1}p(r))\leq \lambda d(T^{r-1}p(r),T^rp(r))\leq \lambda^r d(p(r),Tp(r)),
\end{eqnarray*}
which is a contradiction.
Now, for any pairwise distinct points $x_n, n=1,2,\ldots,k$, we have
\begin{eqnarray*}
P(Tx_1,Tx_2,\ldots,Tx_{k-1},Tx_k)\leq \lambda P(x_1,x_2,\ldots,x_{k-1},x_k).
\end{eqnarray*}
This shows that $T$ is a perimetric contraction on $k$-polygons in $X$. By Theorem \ref{th:2.2}, $T$ has a maximum of $k-1$ fixed points in $X$. The Banach contraction condition implies the uniqueness of the fixed point.
\end{proof}

\begin{lemma}\label{lem:2.3} Let $(X, d)$ be a metric space with $|X|\geqslant 3$, and let $T:X\rightarrow X$ be
a perimetric contraction on $k$-polygons, $3\leqslant k\leq|X|$. If $x$ is an accumulation point of $X$, then Banach contraction inequality \eqref{eq:2.2} holds for all points $y\in X$.
\end{lemma}
\begin{proof} Given any accumulation point $x\in X$, and any $y\in X$. If $x=y$, then \eqref{eq:2.2} holds trivially.\\
Assume that $x\neq y$, since $x$ is an accumulation point of $X$, there exists a sequence $x_n\rightarrow x$ such that $x_n\neq x, x_n\neq y$ and all $x_n$ are pairwise distinct. Hence, by \eqref{eq:2.1}, we have
\begin{equation}\label{eq:2.3}
\begin{split}
&P(Tx,Ty,Tx_n,Tx_{n+1},\ldots,Tx_{n+k-4},Tx_{n+k-3})\\
&\leq\lambda P(x,y,x_n,x_{n+1},\ldots,x_{n+k-4},x_{n+k-3}),
\end{split}
\end{equation}
for all $n\in\mathbb{N}$. Since $d(x_n,x)\rightarrow 0$ and metric  $d$ is continuous, we have $d(y,x_n)\rightarrow d(y,x)$. Also, by Proposition \ref{pro:2.2}, every perimetric contraction on $k$-polygons is continuous, we have $d(Ty,Tx_n)\rightarrow d(Tx,Tx)$. Letting $n\rightarrow\infty$ in \eqref{eq:2.3}, we have
\begin{eqnarray*}
2d(Tx,Ty)\leq2\lambda(d(x,y)+d(y,x)),
\end{eqnarray*}
which is equivalent to \eqref{eq:2.2}.
\end{proof}
The following proposition is a direct corollary of Lemma~\ref{lem:2.3}.
\begin{proposition}\label{cor:2.2} Let $(X,d)$ be a metric space with $|X|\geqslant 3$, let $T:X\rightarrow X$ be a perimetric contraction on $k$-polygons, $3\leqslant k\leq|X|$, $k\in\mathbb{N}$. If all points in $X$ are accumulation points, then $T$ is a Banach contraction mapping.
\end{proposition}

\section{Kannan-Type Perimetric Contraction on $k$-Polygons and Related Fixed Point Theorems}
\quad In this section we introduce the Kannan-type perimetric contraction on $k$-polygons and prove the a fixed point theorem for such mappings.
\begin{definition}\label{def:3.1} Let $(X,d)$ be a metric space with $|X|\geqslant 3$. A mapping $T:X\rightarrow X$ is said to be a Kannan-type perimetric contraction on $k$-polygons $(3\leqslant k\leq|X|, k\in\mathbb{N})$ in $X$ if there exits $\mu\in[0,\frac{2}{k})$ such that
\begin{eqnarray}\label{eq:3.1}
P(Tx_1,Tx_2,\ldots,Tx_{k-1},Tx_k)\leq\mu(d(x_1,Tx_1)+d(x_2,Tx_2)+\cdots+d(x_k,Tx_k))
\end{eqnarray}
for all pairwise distinct points $x_i\in X, i=1,2,\ldots,k, k\geqslant 3$.
\end{definition}
\begin{remark}\label{rm:3.1} If we choose $k=3$ in Definition \ref{def:3.1}, then the mapping $T$ coincides with the notion of generalized Kannan type mapping introduce by E. Petrov, and R.K. Bisht \cite{Petrov2}.
\end{remark}

\begin{proposition}\label{pro:3.1} Every Kannan-type mapping whose contraction coefficient $\gamma$ lies in $[0,\frac{1}{k})$ $(k\in\mathbb{N},k\geqslant 3)$ is a Kannan-type perimetric contraction on $k$-polygons.
\end{proposition}
\begin{proof} Let $(X,d)$ be a metric space with $|X|\geqslant 3$, and $T:X\rightarrow X$ be a Kannan type mapping. For any pairwise distinct points $x_i\in X, i=1,2,\ldots,k$, we have
\begin{align*}
&d(Tx_1,Tx_2)\leq\gamma(d(x_1,Tx_1)+d(x_2,Tx_2)),\\
&d(Tx_2,Tx_3)\leq\gamma(d(x_2,Tx_2)+d(x_3,Tx_3)),\\
&\quad\vdots\\
&d(Tx_{k-1},Tx_k)\leq\gamma(d(x_{k-1},Tx_{k-1})+d(x_k,Tx_k)),\\
&d(Tx_1,Tx_k)\leq\gamma(d(x_1,Tx_1)+d(x_k,Tx_k)).
\end{align*}
Adding the left and right sides of the above inequalities, we have
\begin{align*}
&P(Tx_1,Tx_2,\ldots,Tx_{k-1},Tx_k)
\\
&\leq 2\gamma (d(x_1,Tx_1)+d(x_2,Tx_2)+\cdots
+d(x_{k-1},Tx_{k-1})+d(x_k,Tx_k)).
\end{align*}
Hence, the desired assertion is concluded.
\end{proof}

\begin{lemma}\label{lem:3.2} Let $(X,d)$ be a metric space with $|X|\geqslant 3$, and let $T:X\rightarrow X$ be a Kannan-type perimetric contraction on $k$-polygons $(3\leqslant k\leq|X|, k\in\mathbb{N})$. If $z$ is an accumulation point of $X$ and $T$ is continuous, then the inequality
\begin{equation}\label{eq:3.11}
2d(Tz,Ty)\leq\mu((k-1)d(z,Tz)+d(y,Ty)).
\end{equation}
holds for all $y\in X$.
\end{lemma}
\begin{proof}
Given any accumulation point $z\in X$, and any $y\in X$. If $z=y$, then \eqref{eq:3.11} holds trivially.\\
Assume that $z\neq y$, since $z$ is an accumulation point of $X$, there exists a sequence $x_n\rightarrow z$ such that $x_n\neq z, x_n\neq y$ and all $x_n$ are pairwise distinct. Hence, by \eqref{eq:3.1}, we have
\begin{align}\label{eq:3.3}
&P(Tz,Ty,Tx_n,Tx_{n+1},\ldots,Tx_{n+k-3})\notag\\
&\leq\mu(d(z,Tz)+d(y,Ty)+d(x_n,Tx_n)+\cdots+d(x_{n+k-3},Tx_{n+k-3})),
\end{align}
for all $n\in\mathbb{N}$. Since $d(x_n,z)\rightarrow 0$ and metric function $d$ is continuous, we have $d(y,x_n)\rightarrow d(y,z)$. Also, due to the continuity of $T$, we have $Tx_n\to Tz$, $d(Ty,Tx_n)\rightarrow d(Ty,Tz)$ and $d(x_n,Tx_n)\rightarrow d(z,Tz)$. Letting $n\rightarrow\infty$ in \eqref{eq:3.3}, we have \eqref{eq:3.11}.
\end{proof}

\begin{proposition}\label{pro:3.3} Let $(X,d)$ be a metric space with $|X|\geqslant 3$, and $T: X\rightarrow X$ be a continuous Kannan-type perimetric
contraction on $k$-polygons $(3\leqslant k\leq|X|, k\in\mathbb{N})$. Suppose that all points of $X$ are accumulation points. Then $T$ is a Kannan-type contraction mapping.
\end{proposition}
\begin{proof}
By Lemma~\ref{lem:3.2}, in addition to~\eqref{eq:3.11} we also get the inequality
\begin{eqnarray}\label{eq:1.15}
2d(Ty,Tz)\leq\mu((k-1)d(y,Ty)+d(z,Tz))
\end{eqnarray}
for all $y,z\in X$. Adding the left and right sides of inequalities~\eqref{eq:3.11} and ~\eqref{eq:1.15} together with $\mu\in[0,\frac{2}{k})$, we get
\begin{align*}
&d(Tz,Ty)\leq\frac{k\mu}{4}((dz,Tz)+d(y,Ty))\\
&= \lambda((dz,Tz)+d(y,Ty)),
\end{align*}
where $\lambda \in [0,\frac{1}{2})$, which completes the proof.
\end{proof}

\begin{proposition}\label{pro:3.4} Let $(X,d)$ be a metric space with $|X|\geqslant 3$, and $T: X\rightarrow X$ be a mapping
contracting perimeters of $k$-polygons $(3\leqslant k\leq|X|, k\in\mathbb{N})$ with $0\leq\lambda<\frac{1}{k+1}$. Then $T$ is a Kannan-type perimetric
contraction on $k$-polygons $(3\leqslant k\leq|X|, k\in\mathbb{N})$ with respect to the metric $d$.
\end{proposition}
\begin{proof} 
By~\eqref{eq:2.1} for all pairwise distinct points $x_i\in X, i=1,2,\ldots,k$ we have
\begin{align*}
&P(Tx_1,Tx_2,Tx_3,\ldots,Tx_{k-1},Tx_k)\\
&\leq\lambda P(x_1,x_2,x_3,\ldots,x_{k-1},x_k)\\
\end{align*}
Using the triangle inequalities $d(x_i,x_j)\leqslant d(x_i,Tx_i)+d(Tx_i,Tx_j)+d(Tx_j,x_j)$, we have
\begin{align*}
&\leq\lambda(2(d(x_1,Tx_1)+d(x_2,Tx_2)+\cdots+d(x_k,Tx_k))\\
&+d(Tx_1,Tx_2)+d(Tx_3,Tx_4)+\cdots+d(Tx_{k-1},Tx_k)+d(Tx_k,Tx_1))\\
&=2\lambda(d(x_1,Tx_1)+d(x_2,Tx_2)+\cdots+d(x_k,Tx_k))\\
&+\lambda P(Tx_1,Tx_2,Tx_3,\ldots,Tx_{k-1},Tx_k).
\end{align*}
Rearranging the above inequality yields
\begin{align*}
&P(Tx_1,Tx_2,Tx_3,\ldots,Tx_{k-1},Tx_k)\\
&\leq \frac{2\lambda}{1-\lambda}(d(x_1,Tx_1)+d(x_2,Tx_2)+\cdots+d(x_k,Tx_k)).
\end{align*}
Since $0\leq\lambda<\frac{1}{k+1}$, $\mu=\frac{2\lambda}{1-\lambda}\in[0,\frac{2}{k})$. Hence, $T$ is a Kannan-type perimetric contraction on $k$-polygons.
\end{proof}

In the next, we establish a condition for the existence of fixed point(s) for Kannan-type perimetric contraction on $k$-polygons.

\begin{theorem}\label{th:3.1} Suppose $(X,d)$ is a complete metric space with $|X|\geqslant 3$. Let $T: X\rightarrow X$ be a Kannan-type perimetric contraction on $k$-polygons $(3\leqslant k\leq|X|, k\in\mathbb{N})$ in $X$. $T$ has a fixed point in $X$ if it does not have periodic points of prime periods $i, i=2,3,\cdots,k-1$, and it can admit at most $k-1$ fixed points.
\end{theorem}
\begin{proof} Let $T: X\rightarrow X$ be a perimetric contraction on $k$-polygons in $X$ that does not have periodic points of prime period $i, i=2,3,\cdots,k-1$.\\
For any chosen $x_0\in X$, define the sequence $\{x_n\}$ by $x_n=Tx_{n-1}, n\in\mathbb{N}$. If $x_n$ is a fixed point of $T$ for any $n\in\mathbb{N}\cup\{0\}$, then the proof is completed. Assume that $x_n\neq Tx_n$ for all $n\in\mathbb{N}\cup\{0\}$, we have $x_n\neq x_{n+1}, n=0,1,2,\cdots$. Since $T$ does not attain periodic points of prime periods $2,3,\cdots,k-1$, therefore, it follows from a simple computation that every $k$ consecutive elements of $\{x_n\}$ are pairwise distinct.\\
Now, for any $n\in\mathbb{N}\cup\{0\}$, we have
\begin{align*}
&d(Tx_n,Tx_{n+1})+d(Tx_{n+1},Tx_{n+2})+\cdots+d(Tx_{n+k-2},Tx_{n+k-1})+d(Tx_{n+k-1},Tx_n)\\
&\quad\leq\mu(d(x_n,Tx_n)+d(x_{n+1},Tx_{n+1})+\cdots+d(x_{n+k-2},Tx_{n+k-2})+d(x_{n+k-1},Tx_{n+k-1})).\\
&\Rightarrow d(x_{n+1},x_{n+2})+d(x_{n+2},x_{n+3})+\cdots+d(x_{n+k-1},x_{n+k})+d(x_{n+k},x_{n+1})\\
&\quad\leq\mu(d(x_n,x_{n+1})+d(x_{n+1},x_{n+2})+\cdots+d(x_{n+k-2},x_{n+k-1})+d(x_{n+k-1},x_{n+k})).\\
&\Rightarrow (1-\mu)d(x_{n+k-1},x_{n+k})\leq\mu(d(x_n,x_{n+1})+d(x_{n+1},x_{n+2})+\cdots+d(x_{n+k-2},x_{n+k-1}))\\
&\quad\quad-(d(x_{n+1},x_{n+2})+d(x_{n+2},x_{n+3})+\cdots+d(x_{n+k},x_{n+1})).\\
&\Rightarrow (1-\mu)d(x_{n+k-1},x_{n+k})\leq\mu(d(x_n,x_{n+1})+d(x_{n+1},x_{n+2})+\cdots+d(x_{n+k-2},x_{n+k-1}))\\
&\quad\quad-d(x_{n+k-1},x_{n+k}).\\
&\Rightarrow (2-\mu)d(x_{n+k-1},x_{n+k})\leq\mu(d(x_n,x_{n+1})+d(x_{n+1},x_{n+2})+\cdots+d(x_{n+k-2},x_{n+k-1})).\\
&\Rightarrow d(x_{n+k-1},x_{n+k})\leq\frac{\mu}{2-\mu}(d(x_n,x_{n+1})+d(x_{n+1},x_{n+2})+\cdots+d(x_{n+k-2},x_{n+k-1})).\\
&\Rightarrow d(x_{n+k-1},x_{n+k})\leq\frac{(k-2)\mu}{2-\mu}\max\{d(x_n,x_{n+1}),d(x_{n+1},x_{n+2}),\ldots,d(x_{n+k-2},x_{n+k-1})\}.
\end{align*}
Denote $\rho=\frac{(k-2)\mu}{2-\mu}$. Then $\rho\in[0,1)$ as $\mu\in[0,\frac{2}{k}), k\in\mathbb{N},k>3$.\\
Let $r_n=d(x_n,x_{n+1}),n\in\mathbb{N}\cup\{0\}$ and $R=\max\{r_1,r_2,\ldots,r_{k-1}\}$.
Then, from the last inequality above, for any $n\in\mathbb{N}\cup\{0\}$, we have
\begin{align*}
&d(x_{n+k-1},x_{n+k})\leq\rho\max\{d(x_n,x_{n+1}),d(x_{n+1},x_{n+2}),\ldots,d(x_{n+k-2},x_{n+k-1})\}\\
&\Rightarrow r_{n+k-1}\leq\rho\max\{r_n,r_{n+1},\ldots,r_{n+k-2}\}.
\end{align*}
Therefore, we have
\begin{eqnarray*}
r_1\leq R, r_2\leq R,\cdots,r_{k-1}\leq R,r_k\leq\rho R,r_{k+1}\leq \rho R,\ldots,r_{2k-2}\leq\rho R,r_{2k-1}\leq\rho^2 R,\cdots.
\end{eqnarray*}
Since $\rho<1$, so we have
\begin{align*}
&r_1\leq R, r_2\leq R,\cdots,r_{k-1}\leq R,r_k\leq\rho^{\frac{1}{k-1}}R,r_{k+1}\leq\rho^{\frac{2}{k-1}}R,\ldots,r_{2k-2}\leq\rho R,r_{2k-1}\leq\rho^{\frac{k}{k-1}} R,\cdots.\\
&\Rightarrow r_n\leq \rho^{\frac{n}{k-1}-1}R,\quad\text{for\ \ all}\quad n\in\mathbb{N}\quad\text{with}\quad n\geq k.
\end{align*}
For any $n\in\mathbb{N}\cup\{0\}$ and for any $m\in\mathbb{N}$, we have
\begin{align*}
d(x_n,x_{n+m})&\leq d(x_n,x_{n+1})+d(x_{n+1},x_{n+2})+\cdots+d(x_{n+m-1},dx_{n+m})\\
&\leq r_n+r_{n+1}+\cdots+r_{n+m-1}\\
&\leq R(\rho^{\frac{n}{k-1}-1}+\rho^{\frac{n+1}{k-1}-1}+\cdots+\rho^{\frac{n+m-1}{k-1}-1})\\
&\leq R\rho^{\frac{n}{k-1}-1}\frac{1-\rho^{\frac{m}{k-1}}}{1-\rho^{\frac{1}{k-1}}}.
\end{align*}
Hence, $d(x_n,x_{n+m})\rightarrow 0$ as $n\rightarrow \infty$ for any $m\in\mathbb{N}$. This implies that $\{x_n\}$ is a Cauchy sequence converging a point $w\in X$ due to the completeness of $X$.\\

Recall that any $k$ consecutive element of the sequence $\{x_n\}$ are pairwise distinct. Note also  that there is not fixed points in the sequence $\{x_n\}$ but the fact that $w$ is fixed is not established yet. Hence, if $w\neq x_i$ for all $i\in \{1,2,...\}$, then inequality~\eqref{eq:3.1} holds for the $k$ pairwise distinct points  $w$, $x_{n-1}$,  $x_n$,...,$x_{n+k-3}$.

Suppose that there exists the smallest possible $i\in \{1,2,...\}$ such that $w=x_i$.  If there exists $m>i$ such that $w=x_m$, then the sequence $\{x_n\}$ is cyclic starting from $i$ and can not be a Cauchy sequence. Hence, the points $w$, $x_{n-1}$,  $x_n$,...,$x_{n+k-3}$ are pairwise distinct at least when $n-1>i$.

Let us prove that $Tw=w$. If there exists $i\in \{1,2,...\}$ such that $x_i=w$, then suppose that $n-1>i$. By the triangle inequality and by inequality~\eqref{eq:3.1} we have
\begin{align*}
&d(w,Tw)\leq d(w,x_{n})+d(x_{n},Tw)=d(w,x_{n})+d(Tx_{n-1},Tw)\\
&\leq d(w,x_n)+P(Tw,Tx_{n-1},Tx_{n},Tx_{n+1},\ldots,Tx_{n+k-4},Tx_{n+k-3})\\
&\leq d(w,x_n)+\mu(d(w,Tw)+d(x_{n-1},Tx_{n-1})+\cdots+d(x_{n+k-3},Tx_{n+k-3})).
\end{align*}
This implies that
\begin{eqnarray*}
(1-\mu)d(w,Tw)\leq d(w,x_n)+\mu(d(x_{n-1},Tx_{n-1})+\cdots+d(x_{n+k-3},Tx_{n+k-3})).
\end{eqnarray*}
Taking the limit in the above inequality as $n\rightarrow \infty$, we get $d(w,Tw)=0$ that is, $Tw=w$.

Assume that $T$ has at least $k$ distinct fixed points, say $w_i, i=1,2,\ldots,k$, that is, $Tw_i=w_i, i=1,2,\ldots,k$.\\
Then by~\eqref{eq:3.1} we have
\begin{align*}
&P(w_1,w_2,...,w_k)\\
&=P(Tw_1,Tw_2,...,Tw_k)\\
&\leq \mu(d(w_1,Tw_1)+d(w_2,Tw_2)+\cdots+d(w_{k-1},Tw_{k-1})+d(w_k,Tw_k))=0,
\end{align*}
which contradicts the fact that $w_i, i=1,2,\ldots,k$ are pairwise distinct. Thus, $T$ has at most $k-1$ fixed points.
\end{proof}

\begin{proposition}\label{pro:3.5} Suppose that under the assumption of Theorem \ref{th:3.1}, the mapping $T$ has a fixed point $w$ that acts as the limit for a specific iteration sequence $\{x_i\}_0^\infty$ defined by $x_i=Tx_{i-1}, i\in\mathbb{N}\cup\{0\}$ with $w\neq x_i$ for all $i\in\mathbb{N}\cup\{0\}$, then $w$ is the unique fixed point of $T$.
\end{proposition}
\begin{proof} Suppose that $z$ is another fixed point of $T$. Then $z\neq x_n$ for all $n\in\mathbb{N}\cup\{0\}$, otherwise, we have $w=z$. Therefore, $w,z,x_n$ are all distinct for all $n\in\mathbb{N}\cup\{0\}$.
Thus, for all $n\in\mathbb{N}\cup\{0\}$ we have
\begin{align*}
&d(Tw,Tz,Tx_n,...,Tx_{n+k-3})\\
&\leq\mu(d(w,Tw)+d(z,Tz)+d(x_n,Tx_n)+\cdots+d(x_{n+k-3},Tx_{n+k-3})),
\end{align*}
which implies that
\begin{align*}
d(w,z,x_{n+1},...,x_{n+k-2})
\leq\mu(d(x_n,x_{n+1})+\cdots+d(x_{n+k-3},x_{n+k-2})).
\end{align*}
Letting $n\rightarrow \infty$ in the above inequality, we have $2d(w,z)\leq 0$, which contradicts to the fact that $w\neq z$. Therefore, $T$ has a unique fixed point.
\end{proof}

\begin{example}
Let $(X,d)$ be a metric space such that $X=\{x_1,...,x_5\}$ and
$d(x_i,x_j)=1$, $i\neq j$, $1\leqslant i,j\leqslant 4$;
$d(x_i,x_5)=9$,  $1\leqslant i\leqslant 4$.
Define $T\colon X\to X$ as follows $Tx_1=x_2$, $Tx_2=x_3$, $Tx_3=x_4$, $Tx_4=x_4$, $Tx_5=x_1$. Further,
$$
P(Tx_1,Tx_2,Tx_3,Tx_4,Tx_5)=P(x_2,x_3,x_4,x_4,x_1)=1+1+0+1+1=4,
$$
$$
d(x_1,Tx_1)+d(x_2,Tx_2)+d(x_3,Tx_3)+d(x_4,Tx_4)+d(x_5,Tx_5)=
$$
$$
d(x_1,x_2)+d(x_2,x_3)+d(x_3,x_4)+d(x_4,x_4)+d(x_5,x_1)=
1+1+1+0+9=12.
$$
Thus, $T$ is a Kannan-type perimetric contraction on $5$-polygons with $\mu=\frac{4}{12}=\frac{1}{3}<\frac{2}{5}=\frac{2}{k}$, $k=5$, see Definition~\ref{def:3.1}, and with the single fixed point $x_4$.
\end{example}

\subsection*{Funding}

This research is partially supported by Key Research and Development Project of Hainan Province (Grant No. ZDYF2023GXJS007); Sanya City Science and Technology Innovation Special Project (Grant No. 2022KJCX22); Talent Program of University of Sanya (Grant No. USYRC19-04) and Key Special Project of University of Sanya (Grant No. USY22XK-04).


\begin{thebibliography}{99}
\bibitem{Banach} S. Banach, Sur les op\'{e}rations dans les ensembles abstraits et leur application aux \'{e}quations int\'{e}grales, {\it Fund. Math.}, {\bf 3}: 133-181, 1922.
\bibitem{Kirk} W.A. Kirk, Contraction mappings and extensions, {\it In: Handbook of Metric Fixed Point Theory pp.1-34}, Kluwer Academic Publishers, Dordrecht, 2001.
\bibitem{Aagrawal1} P. Agarwal, M. Jleli, B. Samet, Fixed Point Theory in Metric Spaces, Springer, Singapore, 2018.
\bibitem{Subrahmanyam} P.V. Subrahmanyam, Elementary Fixed Point Theorems, {\it Forum for Interdisciplinary Mathematics}, Springer, Singapore, 2018.
\bibitem{Reich} S. Reich, Some remarks concerning contraction mappings, {\it Canad. Math. Bull.}, {\bf14(10)}: 121-124, 1971.
\bibitem{Ciric} L.B. \'{C}iri\'{c},: A generalization of Banach's contraction principle, {\it Proc. Amer. Math. Soc.}, {\bf 45}: 267-273, 1974.
\bibitem{Hardy} G.E. Hardy, T.D. Rogers, A generalization of a fixed point theorem of Reich, {\it Canad. Math. Bull.}, {\bf16(2)}: 201-206, 1973.
\bibitem{Berinde} V. Berinde, Iterative Approximation of Fixed Points, Springer, Berlin, 2007.
\bibitem{Rus} I.A. Rus, Generalized Contractions and Applications, Cluj Univ. Press, Cluj-Napoca, 2001.
\bibitem{Cho} Y.J. Cho, M. Jleli, M. Mursaleen, B. Samet, C. Vetro, (eds.), Advances in Metric Fixed Point Theory
and Applications, Springer, Singapore. 2021.
\bibitem{Debnath} P. Debnath, N. Konwar, S. Radenovi\'{c}, (eds.), Metric Fixed Point Theory \& Applications in Science
Engineering and Behavioural Sciences, {\it Forum for Interdisciplinary Mathematics}, Springer, Singapore, 2021.
\bibitem{Pittnauer} F. Pittnauer, Ein fixpunktsatz in metrischen r\"{a}umen, {\it Arch. Math.} {\bf 26}: 421-426, 1975.
\bibitem{Achari} J. Achari, Fixed point theorems in complete metric spaces, {\it Results Math.}, {\bf 3}: 1-6, 1980.
\bibitem{Petrov1} E. Petrov, Fixed point theorem for mappings contracting perimeters of triangles, {\it J. Fixed Point Theory Appl.}, {\bf 25}:74, 2023.
\bibitem{Petrov2} E. Petrov and R.K. Bisht, Fixed point theorem for generalized Kannan type mappings, {\it Rend. Circ. Mat. Palermo, II.}, {\bf 2024}: 1-18, 2024.
\bibitem{Petrov3} E. Petrov, Periodic points of mappings contracting total pairwise distance, {\it Arxiv}, arXiv:2402.02536, 2024.
\bibitem{Kannan} R. Kannan, Some results on fixed points, {\it Bull. Calcutta Math. Soc.}, {\bf 60}: 71-76, 1968.
\bibitem{Chatterjea} S.K. Chatterjea, Fixed-point theorems, {\it C. R. Acad. Bulgare Sci.}, {\bf 25}: 727-730, 1972.
\bibitem{Anis} B. Anis, M. Pratikshan, K.D. Lakshmi, Perimetric Contraction on Quadrilaterals and Related Fixed Point Results, {\it Arxive},	arXiv:2407.07542, 2024.
\bibitem{Berinde2} V. Berinde, Approximating fixed points of weak contractions using the Picard iteration, {\it Nonlinear Analysis Forum}, {\bf 9}: 43--53, 2004.
\bibitem{Wardowski} D. Wardowski, Fixed point theory of a new type of contractive mappings in complete metric spaces, {\it Fixed Point Theory Appl.}, {\bf 2012}:94, 2012.
\bibitem{Berinde3} V. Berinde, M. P\u{a}curar, Approximating fixed points of enriched contractions in Banach
spaces. {\it J. Fixed Point Theory Appl.}, {\bf 22}: 1-10, 2020.
\end{thebibliography}
\end{document}